\documentclass{IEEEtaes}

\usepackage{color,array,amsthm}
\usepackage{graphicx}
\usepackage{subcaption}
\usepackage{todonotes}
\usepackage{csquotes}
\usepackage{amsmath} 
\usepackage{tabularx} 
\usepackage{hyperref}        
\usepackage{orcidlink}       

\jvol{XX}
\jnum{XX}
\jmonth{XXXXX}
\paper{1234567}
\pubyear{2025}
\doiinfo{TAES.2025.Doi Number}

\newtheorem{theorem}{Theorem}

\begin{document}

\title{A Path Planning Algorithm for a Hybrid UAV Traveling in Noise Restricted Zones} 
\markboth{IEEE Transactions on Aerospace and Electronic Systems. Preprint Version.}{IEEE Transactions on Aerospace and Electronic Systems. Preprint Version.}

\author{Saurabh Belgaonkar\, \orcidlink{0000-0002-7628-4930}}
\author{Deepak Prakash Kumar \, \orcidlink{0000-0001-6446-2077}, ~\IEEEmembership{Student Member,~IEEE}}

\author{Sivakumar Rathinam\, \orcidlink{0000-0002-9223-7456},  ~\IEEEmembership{Senior Member,~IEEE}}

\author{Swaroop Darbha\, \orcidlink{0000-0001-8377-0657}, ~\IEEEmembership{Fellow,~IEEE}}

\author{Trevor Bihl\, \orcidlink{0000-0003-2431-2749},  ~\IEEEmembership{Senior Member,~IEEE}}

\receiveddate{}

\corresp{\itshape \ Corresponding author: Saurabh Belgaonkar.}

\authoraddress{
Saurabh Belgaonkar (\href{mailto:saurabhbelgaonkar@tamu.edu}{saurabhbelgaonkar@tamu.edu}), Deepak Prakash Kumar (\href{mailto:deepakprakash1997@gmail.com}{deepakprakash1997@gmail.com}), and Swaroop Darbha (\href{mailto:dswaroop@tamu.edu}{dswaroop@tamu.edu}) are affiliated with the Mechanical Engineering Department of Texas A\&M University, College Station. Sivakumar Rathinam (\href{mailto:srathinam@tamu.edu}{srathinam@tamu.edu}) is affiliated with the Computer Science and Engineering Department and the Mechanical Engineering Department of Texas A\&M University, College Station.
Trevor Bihl (\href{mailto:trevor.bihl.2@afrl.af.mil}{trevor.bihl.2@afrl.af.mil}) is affiliated with Air Force Research Laboratory (AFRL), Dayton.
}
\markboth{Belgaonkar ET AL.}{A Path Planning Algorithm for a HUAV Traveling in Noise Restricted Zones}
\maketitle
\begin{abstract}
This paper presents an integrated approach for efficient path planning and energy management in hybrid unmanned aerial vehicles (HUAVs) equipped with dual fuel-electric propulsion systems. These HUAVs operate in environments that include noise-restricted zones, referred to as quiet zones, where only electric mode is permitted. We address the problem by parameterizing the position of a point along the side of the quiet zone using its endpoints and a scalar parameter, which transforms the problem into a variant of finding the shortest path over a graph of convex sets. We formulate this problem as a mixed-integer convex program (MICP), which can be efficiently solved using commercial solvers. Additionally, a tight lower bound can be obtained by relaxing the path-selection variable. Through extensive computations across 200 instances over four maps, we show a substantial improvement in computational efficiency over a state-of-the-art method, achieving up to a 100-fold and 10-fold decrease in computation time for calculating the lower bound and the exact solution, respectively. Moreover, the average gap between the exact cost and the lower bound was approximately 0.24\%, and the exact cost was 1.05\% lower than the feasible solution obtained from the state-of-the-art approach on average, highlighting the effectiveness of our proposed method. We also extend our approach to plan the route of the HUAV to visit a set of targets and return to its starting location in environments with quiet zones, yielding a Traveling Salesman Problem (TSP). We employ two methodologies to solve the TSP: one wherein the battery's State of charge (SOC) at each target is discretized, whereas, in another approach, the battery's SOC at each target is assumed to be the minimum allowable level when the HUAV departs it. A comparative analysis reveals that the second method achieves a cost within 1.02\% of the first method on average while requiring significantly less computational time.
\end{abstract}

\begin{IEEEkeywords}
Graph of convex set (GCS), Hybrid Unmanned Vehicle (HUAV), Noise-restricted Zones, Mixed integer convex program (MICP), Path Planning, Traveling salesman problem (TSP)
\end{IEEEkeywords}

\section{INTRODUCTION}
In the field of urban air mobility (UAM), electric propulsion is the predominant technology, particularly for applications such as drone delivery \cite{goyal2018urban,cohen2021urban}. However, the limited flight time of electric unmanned aerial vehicles (UAVs) poses a significant barrier to their potential for reducing carbon emissions \cite{button2021faith, mohsan2022towards}. Hybridization, which combines fuel engines with electric propulsion, provides a promising solution by extending the operational range and flight time of these vehicles \cite{townsend2020comprehensive, kusmierek2023review, tao2019state, farajollahi2022hybrid}. 
Various efforts have been made to study the impacts of UAV sound on humans, including noise-induced hearing loss \cite{kronoupsiloneter1970airplane} and annoyance from drone usage in urban centers \cite{watkins2020ten}. As UAM evolves, particularly in densely populated urban areas, managing the noise pollution generated by fuel engines becomes an increasingly critical concern \cite{rizzi2020urban,cohen2021urban,kusmierek2023review}. 

In this paper, a series-hybrid UAV \cite{lieh2011design, an2024design} is considered, where propellers are driven by electric motors powered either by a fuel engine generator or a battery. Although the fuel engine extends the UAV's range, it significantly contributes to noise, which is a major concern for applications such as residential deliveries and military surveillance \cite{kusmierek2023review, dombrovschi2024acoustic}. Operating solely in electric mode---thereby eliminating the use of the fuel engine--- reduces noise emissions \cite{kim2018review, dombrovschi2024acoustic}. Consequently, the hybrid configuration offers a balance between extended flight time and reduced noise emission. The HUAV operates in two modes: 1) In the {\it electric mode}, propellers are powered only by battery, 2) in the {\it fuel mode}, fuel is used to power the propellers and charge the battery. We assume that the system allows for seamless switching between fuel and electric modes, with the activation of the fuel engine occurring at zero cost. Furthermore, the HUAVs considered in this study are capable of making sharp turns, allowing us to neglect any kinematic constraints in the path planning.
\begin{figure}[htb!]
 \begin{subfigure}[b]{0.75\linewidth}
     \centering
     \includegraphics[width=\textwidth]{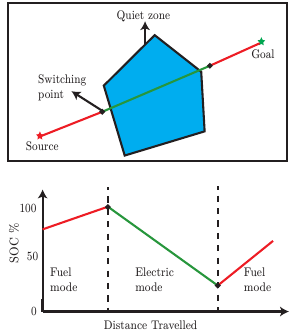}
     \caption{A feasible path between two locations alongside its corresponding battery charge profile.}
     \label{subfig: feasible_path}
 \end{subfigure}
 \hspace{0.05\linewidth}
 \begin{subfigure}[b]{0.75\linewidth}
     \centering
     \includegraphics[width=\textwidth]{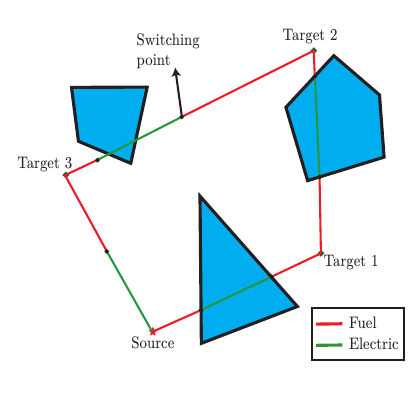}
     \caption{A feasible path for a TSP, including mode switching points.}
     \label{subfig: feasible_path_TSP}
 \end{subfigure}
    \caption{Example of feasible paths (a) for a path planning problem between two locations, (b) for a Traveling Salesman Problem (TSP) involving multiple locations.}
    \label{fig:Feasible_path_figs}
\end{figure}

The objective is not only to determine an optimal path for the HUAV but also to develop an optimal strategy for utilizing the two propulsion modes, minimizing fuel consumption while adhering to the vehicle battery's operational limits during navigation. A key scenario involves quiet zones, such as residential neighborhoods, where the use of gas engines is restricted due to their elevated noise levels. Fig. \ref{subfig: feasible_path} illustrates how the HUAV switches to electric-only mode in these zones, and how the battery charges during fuel-only mode and discharges during the electric-only mode. This problem can be generalized to finding an optimal tour for a HUAV that visits multiple designated targets as illustrated in Fig. \ref{subfig: feasible_path_TSP}, minimizing fuel consumption while considering operational constraints such as maintaining the battery's SOC and complying with noise restrictions in quiet zones. This generalization can be viewed as a Traveling Salesman Problem (TSP) for the HUAV traveling in noise-restricted zones. 

Although quiet zones might initially be considered obstacles in HUAV path-planning algorithms, this analogy doesn't fully capture the unique challenges involved. Traditional methods—such as Rapidly-exploring Random Trees (RRT) \cite{karaman2011anytime}, Probabilistic Roadmap Methods (PRM) \cite{geraerts2004comparative}, Voronoi diagrams \cite{aurenhammer1991voronoi}, and potential fields \cite{vadakkepat2000evolutionary}—focus on navigating around physical obstacles to ensure collision-free paths. However, our work differs significantly from conventional obstacle avoidance. We emphasize enforcing the restriction to electric mode within designated quiet zones to minimize noise instead of avoiding the quiet zone. 

Previous work \cite{adlakha2023integration} integrates noise awareness into grid-based UAV path planning, while \cite{sarhan2025noise} employs reinforcement learning to minimize UAV noise in urban environments. However, these approaches focus solely on conventional UAVs without considering propulsion mode switching. In contrast, this work addresses noise-aware path planning for a hybrid UAV, explicitly accounting for transitions between fuel and electric propulsion modes, which adds complexity not considered in previous studies.

Planning in the presence of operational constraints imposed by quiet zones is challenging, as the specific quiet zones to traverse, along with their entry and exit points, are not known a priori. One approach to address this challenge involves discretizing the boundaries of quiet zones and the battery SOC, thereby transforming the problem into a graph-based discrete planning problem \cite{manyam2022path}. In this formulation, nodes represent discretized spatial points and SOC levels, while the feasibility and cost of edges are computed subsequently. Dijkstra's algorithm is then used to determine the least-cost path, identifying the optimal route and the switching points between battery and fuel modes. While this method is effective, it becomes computationally intensive for large-scale maps with multiple quiet zones due to the significant level of discretization required. 

A similar problem was tackled in \cite{scott2024noise}, denoted as The Noise-Restricted Hybrid-Fuel Shortest Path Problem (NRHFSPP). In \cite{scott2024noise}, the authors focused on coupled path and power planning for hybrid UAVs in noise-restricted environments. Their discrete optimization framework, which is based on a labeling algorithm, efficiently computes the optimal path and generator schedule on large graphs. However, the approach prohibits mode switching along a single edge, limiting flexibility in scenarios where long edges may require switching. Additionally, solving the problem in a discretized position space can lead to the creation of large graphs, significantly increasing computational time.

Unlike previous methods that discretize the battery's SOC and the hybrid UAV's (HUAV) positions along quiet-zone boundaries, our approach addresses the problem directly in continuous space. The key idea is to parameterize positions along the boundaries of the quiet zones—modeled as closed convex sets—using their endpoints and a scalar parameter. This reformulation transforms the problem into a variant of finding the shortest path over a graph of convex sets \cite{marcucci2024shortest}.

By incorporating SOC constraints and a cost based on the distance traveled in fuel mode, we develop a mixed-integer bilinear formulation. Through appropriate reformulation of the bilinear terms, the problem is expressed as a mixed-integer convex program (MICP) that can be efficiently solved using commercial solvers. This approach not only identifies the optimal path but also determines the optimal mode-switching points along the route.
In addition, we can obtain a tight lower bound to the path-planning problem by relaxing the path-selection variable in the formulated MICP. Numerical experiments, using the discrete approach from \cite{manyam2022path} as a baseline, reveal that our method significantly outperforms the discrete approach in terms of optimization time. We also generalize our approach to solve the TSP for a HUAV traveling in noise-restricted zones.
The structure of this paper is as follows: Section \ref{sec:prob_statement} outlines the problem statement. The mixed integer program and its reformulation are presented in sections \ref{sec:MIBP} and \ref{sec:RMICP} respectively. Section \ref{sec:TSP} extends the approach to solve the TSP. Finally, results are presented in section \ref{sec:Results} with concluding remarks in section \ref{sec:conclusions}.

\section{Problem Statement}\label{sec:prob_statement}

We assume that the HUAV operates in a two-dimensional ($2\mathrm{D}$) plane, where quiet zones are represented as convex polygons.\footnote{If a quiet zone is not convex, we first compute its convex hull to approximate it as a convex polygon.} Since there is only one HUAV, it will also be referred to as \enquote{the vehicle} in this paper. Let the vehicle travel at a constant speed along any assigned path from the source ($s$) to the destination ($g$); hence, the distance traversed by the vehicle is directly proportional to its travel time. Let the SOC of the vehicle at a given position $x$ be denoted by $q(x)$. For any position $x$ along the vehicle's path, we require $q(x)$ to lie within specified bounds, i.e., $q_{min} \leq q(x) \leq q_{max}$. If the vehicle is traveling in {\it fuel mode}, the rate of SOC depletion per unit distance is denoted by $\alpha$. If the vehicle is traveling in {\it electric mode}, the rate of SOC replenishment per unit distance is denoted by $\beta$. 

\begin{figure}
    \centering
    \includegraphics[width=1\linewidth]{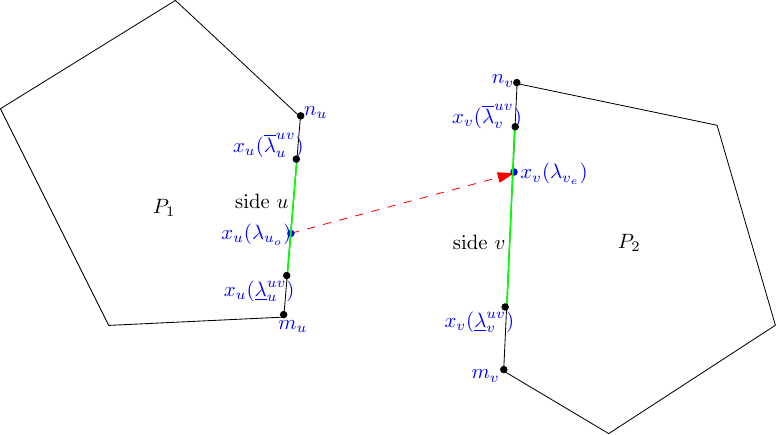}
    \caption{Travel between two sides corresponding to different quiet zones. The green segments of the sides $u$
and $v$ represent the portions that can be directly traveled between without entering any other quiet zone. }
    \label{fig:notations}
\end{figure}

Let \(P = \{Q_1, Q_2, \dots, Q_m\}\) denote the set of quiet zones. Let $\mathcal{V}_s := \{1, 2, \dots, n\}$ denote the set of all sides of the quiet zones in $P$. 
Define $\mathcal{V}:=\mathcal{V}_s \cup \{s, g\}$. The elements of $\mathcal{V}$ are referred to as nodes in our problem. The positions of the endpoints of any side \( v \in \mathcal{V}_s \) is denoted by \( m_v \) and \( n_v \), as shown in Fig. \ref{fig:notations}. The position of any point on the side \(v \) can be expressed using a parameter \(\lambda_v \in [0,1]\) as \(x_v(\lambda_v) = \lambda_v m_v + (1 - \lambda_v) n_v \). 

Given any two distinct sides \( u \) and \( v \) belonging to the same quiet zone, the vehicle can travel directly from any point on \( u \) to any point on \( v \) using only electric mode. On the other hand, if the two sides $u$ and $v$ belong to different quiet zones, we specify parts of the sides that can be directly traversed between each other without entering any other quiet zone. These are defined through the {\it path boundary constraints}. Consider the parameters\footnote{In the appendix, we provide an algorithm for computing these parameters.} $\underline{\lambda}^{uv}_u$ and $\overline{\lambda}^{uv}_u$ corresponding to side $u$, and parameters $\underline{\lambda}^{uv}_v$ and $\overline{\lambda}^{uv}_v$ corresponding to side $v$ (refer to Fig. \ref{fig:notations}). If the vehicle chooses to travel from $u$ to $v$, the {\it path boundary constraints} state that the vehicle must travel from a point in the line segment joining $x_u(\underline{\lambda}^{uv}_u)$ and $x_u(\overline{\lambda}^{uv}_u)$ to a point in the line segment joining  $x_v(\underline{\lambda}^{uv}_v)$ and $x_v(\overline{\lambda}^{uv}_v)$. These constraints allow the vehicle to directly travel between sides of two distinct quiet zones without entering the interior of any other quiet zone. Similar parameters are also defined for the travel between the source (or goal) and any side in $\mathcal{V}_s$.

A path for the vehicle is defined as a sequence of nodes in $\mathcal{V}$ visited by the vehicle. The objective of the path planning problem is to determine (i) a simple path $\mathcal{P}$ for the vehicle that starts at $s$ and ends at $g$, (ii) the position of the vehicle as it enters and exits any node in $\mathcal{P}$, and (iii) the portion of the travel between any two adjacent nodes in $\mathcal{P}$ during which the vehicle operates in the fuel mode, such that the following constraints are satisfied:

\begin{itemize} 
\item {\it SOC constraints:} The state of charge (SOC) of the vehicle at any point along $\mathcal{P}$ remains within the specified bounds. 
\item {\it Quiet zone constraints:} The vehicle operates only in electric mode when traveling within the interior of any quiet zone. 
\item {\it Path boundary constraints:} $\mathcal{P}$ satisfies the path boundary constraints.
\item {\it Distance minimization in fuel mode constraint:} The sum of the Euclidean distances traveled by the vehicle in the fuel mode along $\mathcal{P}$ is minimized. \end{itemize}
An illustration of a feasible path for the HUAV is shown in Fig. \ref{fig:samplepath}.

\begin{figure}
    \centering
    \includegraphics[width=1\linewidth]{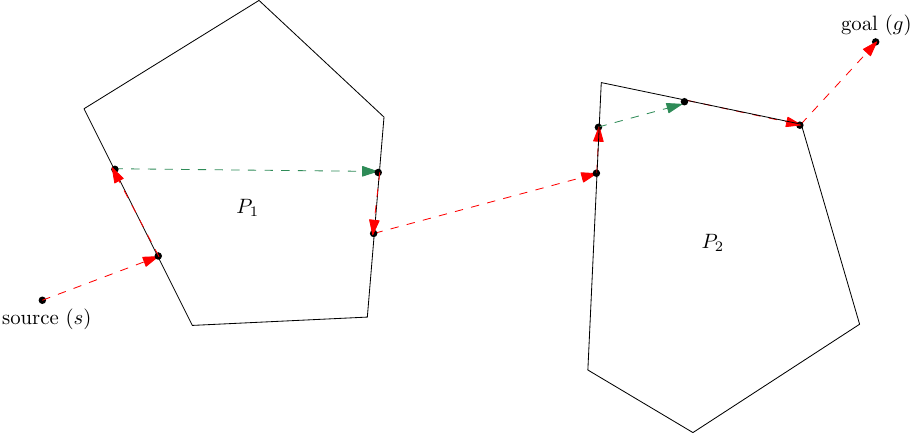}
    \caption{A sample path for the HUAV crossing two quiet zones $P_1$ and $P_2$. The dark green arrows show the portions where the vehicle is traveling in {\it electric mode} and the red arrows show the portions where the HUAV is traveling in {\it fuel mode}.}
    \label{fig:samplepath}
\end{figure}

\section{ Mixed Integer Bilinear Program (MIBP)}\label{sec:MIBP}
The MIBP is formulated on the graph $\mathcal{G}=(\mathcal{V},\mathcal{E})$ where $\mathcal{E}:=\mathcal{E}_{\text{in}} \bigcup \mathcal{E}_{\text{o}}$ consists of two types of edges as follows:

\begin{itemize}    
    \item \textbf{Intra-Quiet Zone Edges (\( \mathcal{E}_{\text{in}} \)):} Any edge \( e = (u,v)\) is in $\mathcal{E}_{\text{in}}$ if and only if \(u, v \in \mathcal{V}_s\) and they connect sides within the same quiet zone. As a result, the vehicle can only travel in the electric mode along any edge in $\mathcal{E}_{\text{in}}$. 

    \item \textbf{Inter-Quiet Zone Edges (\( \mathcal{E}_{\text{o}} \)):} Any edge \( e = (u,v)\) is in $\mathcal{E}_{\text{o}}$ if and only if \(u, v \in \mathcal{V}\) and they either connect sides from distinct quiet zones or connect the start/goal node to a side of a quiet zone. As a result, such edges lie entirely outside the interior of the quiet zones, allowing the use of both fuel and electric modes. 
\end{itemize}

\newcolumntype{L}[1]{>{\raggedright\arraybackslash}p{#1}} 
\begin{table}[h!]
    \caption{Decision variables in MIBP.}
    \label{tab:decvar}
    \centering
    \begin{tabularx}{\columnwidth}{|L{1.7 cm}|L{1.5cm}|L{3.9cm}|}
        \hline
        \textbf{Decision Variable} & \textbf{Variable Type} & \textbf{Definition} \\
        \hline
        $y_{uv}$ for edge $(u,v)\in \mathcal{E}$ & Binary & $y_{uv}=1$ if the vehicle travels from $u$ to $v$; $y_{uv}=0$ otherwise. \\
        \hline
        $w_v$ for $v \in \mathcal{V}_s$ & Binary & $w_v=1$ if the vehicle travels along side $v$; $w_v=0$ otherwise. \\
        \hline
        $z_{uv}$ for edge $(u,v)\in \mathcal{E}$ & Continuous & $z_{uv}$ is the distance traveled in fuel mode from $u$ to $v$. \\
        \hline
        $L_{uv}$ for edge $(u,v)\in \mathcal{E}$ & Continuous & $L_{uv}$ is the total distance traveled from $u$ to $v$. \\
        \hline
        $p_v$ for $v \in \mathcal{V}_s$ & Continuous & $p_v$ is the distance traveled in fuel mode along side $v$. \\
        \hline
        $d_v$ for $v \in \mathcal{V}_s$ & Continuous & $d_v$ is the total distance traveled along side $v$. \\
        \hline
        $q_{v_e}$ for $v \in \mathcal{V}$ & Continuous & The SOC of the vehicle upon entering node $v$. \\
        \hline
        $q_{v_o}$ for $v \in \mathcal{V}$ & Continuous & The SOC of the vehicle upon leaving node $v$. \\
        \hline
        $\lambda_{v_e}$ for $v \in \mathcal{V}_s$ & Continuous & Positional parameter when entering side $v$. \\
        \hline
        $\lambda_{v_o}$ for $v \in \mathcal{V}_s$ & Continuous & Positional parameter when leaving side $v$. \\
        \hline
    \end{tabularx}
\end{table}

All the decision variables used in the MIBP are defined in Table \ref{tab:decvar}. We will first state the obvious requirement on the path variables before formulating each constraint in the problem statement. $y_{uv}$ for edge $(u,v)\in \mathcal{E}$ must represent a valid path from the source to the goal node. This requirement is stated using the following flow constraints:

\begin{align}
    \sum_{(s,v) \in \mathcal{E}}  y_{sv} &= 1 , \
    \sum_{(v,g) \in \mathcal{E}}  y_{vg} = 1.\label{source_sink_constraint} \\
     \sum_{(u,v) \in \mathcal{E}}  y_{vu} &= \sum_{(u,v), \in \mathcal{E}} y_{uv},  \quad  \forall v \in \mathcal{V}_s.\label{flow_constraints}
\end{align}
\begin{align}
     \sum_{(u,v) \in \mathcal{E}} y_{vu}  & \leq w_{v}   \quad  \forall v \in \mathcal{V}_s.\label{flow_constraints_max_one} \\
  y_{uv}\in\{0,1\} ~\forall (u,v) & \in \mathcal{E},  ~~ w_v\in\{0,1\}  ~\forall v \in \mathcal{V}_s.
\end{align}

\subsection{\it SOC constraints:} This constraint requires the SOC of the vehicle at any point along $\mathcal{P}$ remain within the specified bounds. This can be stated using the variables $q_{v_e}$ and $q_{v_o}$ for $v\in \mathcal{V}$. As the vehicle traverses from node \( u \) to \( v \), its SOC changes depending on the portion of its travel in the fuel mode. For any inter-quiet zone edge in $\mathcal{E}_{o}$, the SOC of the vehicle ($q_{v_e}$) when it arrives at node $v$ can be computed as follows: 
\begin{align}
        &  q_{v_e} y_{uv} {=}   q_{u_o} y_{uv} {+}  \beta z_{uv} y_{uv} {-} \alpha (L_{uv} {-} z_{uv})  y_{uv} \  \   \forall (u,v) \in \mathcal{E}_o.    \label{SOC_constraints_o} 
\end{align}
Note that the above constraint is switched on only if $y_{uv}=1$; otherwise it is vacuous. Similarly, for any intra-quiet zone edge in  $\mathcal{E}_{in}$, the SOC of the vehicle ($q_{v_e}$) when it arrives at node $v$ can be obtained as follows: 
\begin{align}
        & q_{v_e} y_{uv} =  q_{u_o} y_{uv}  -  \alpha L_{uv}  y_{uv} \ \  \forall (u,v) \in \mathcal{E}_{in}. \label{SOC_constraints_i} 
\end{align}
We also allow for the vehicle to travel along a side in $\mathcal{V}_s$. To keep track of the SOC of the vehicle during this motion, we have the following constraint:
\begin{align}
        &  q_{v_o} w_{v} =   q_{v_e} w_{v} + \beta z_{v} w_{v}  -  \alpha (d_{v} - p_{v}) w_{v}  \quad  \forall v \in \mathcal{V}_s. \     \label{SOC_constraints_same} 
\end{align}
Once we keep track of the SOC of the vehicle at all  times, we can now state the bounding constraints of the vehicle's SOC as follows: 
\begin{align}
        q_{\text{min}}  \leq  q_{v_e}  \leq q_{\text{max}}, ~q_{\text{min}}  \leq  q_{v_o}  \leq q_{\text{max}}  \quad \ \ \forall v \in \mathcal{V}.
    \label{q_bound}
\end{align}

\subsection{\it Quiet zone constraints:} These constraints are already formulated in \eqref{SOC_constraints_i} as the vehicle is only allowed to travel in the electric mode along any edge in $\mathcal{E}_{in}$.

\subsection{\it Path boundary constraints:} Consider the travel from node $u\in \mathcal{V_s}$ to node $v\in \mathcal{V_s}$. The position of the vehicle when it leaves side $u$ is denoted by $x_u(\lambda_{u_o})$. Similarly, the position of the vehicle when it arrives at side $v$ is denoted by $x_v(\lambda_{v_e})$. The parameters $\lambda_{u_o}$ and $\lambda_{v_e}$ must satisfy the following path boundary constraints:
\begin{align}
    \underline{\lambda}^{uv}_u y_{uv} \leq  \lambda_{u_{o}}  y_{uv} & \leq \overline{\lambda}^{uv}_u y_{uv}, \notag \\
    \underline{\lambda}^{uv}_v y_{uv} \leq  \lambda_{v_{e}}  y_{uv} & \leq
    \overline{\lambda}^{uv}_v y_{uv}.    \label{Lambda_bound} 
\end{align}

Similar constraints can also be stated for travel from the source $s$ to a side $v$ or from a side $u$ to the goal $g$ as follows:
\begin{align}
  \underline{\lambda}^{sv}_v y_{sv} \leq  \lambda_{v_{e}}  y_{sv} & \leq
    \overline{\lambda}^{sv}_v y_{sv}, \\
    \underline{\lambda}^{ug}_u y_{ug} \leq  \lambda_{u_{o}}  y_{ug} & \leq \overline{\lambda}^{ug}_u y_{ug}.   \label{Lambda_source_goal_bound} 
\end{align}

\subsection{\it Distance minimization in fuel mode constraint:} This constraint requires us to first calculate the Euclidean distances traveled by the vehicle in the fuel mode. Consider the travel of the vehicle from $u\in \mathcal{V}_s$ to $v\in \mathcal{V}_s$. The total distance ($L_{uv}$) traveled by the vehicle from $x_u(\lambda_{u_o})$ to $x_v(\lambda_{v_e})$ satisfies the following constraint:
\begin{align}
\forall (u, v) \in \mathcal{E}, ~~ (L_{uv} y_{uv})^2 \geq ~& \| x_u(\lambda_{u_o}) y_{uv}-x_v(\lambda_{v_e})y_{uv} \|_2^2  \nonumber \\
      =  \|n_u y_{uv} {+} \lambda_{u_o} y_{uv}(m_u {-} n_u) ~&{-} n_v y_{uv}
    {-} \lambda_{v_e} y_{uv}(m_v {-} n_v)\|_2^2. \label{eq:intra_length_constraint} 
\end{align}

Since $z_{uv}$ denotes the distance traveled by the vehicle in the fuel mode, it cannot exceed $L_{uv}$. Therefore, 
\begin{align}
        0 \leq  z_{uv} & \leq L_{uv} \quad  \forall (u,v) \in \mathcal{E}.    \label{z_bound}
\end{align}
Note that, even though \eqref{eq:intra_length_constraint} is not a strict equality,  $z_{uv}$ will never exceed the Euclidean distance between $x_u(\lambda_{u_o})$ and $x_v(\lambda_{v_e})$ as $z_{uv}$ would not be optimal otherwise. Similar constraints can also be written when the vehicle travels along a side $v\in \mathcal{V}_s$ as follows:
\begin{align}
    (d_v w_{v})^2 \geq  \|\lambda_{v_o} w_{v} -  \lambda_{v_e} w_{v}   \|^2 \|m_v & - n_v\|_2^2  \ \quad  \forall v \in \mathcal{V}_s. \label{length_constraint_same_side} \\
    0 \leq p_v  \leq d_v, \ \ 0   \leq d_v  \leq  \| m_v - & n_v\|_2^2 \quad \forall v \in \mathcal{V}_s. \label{d_bound_same_side}
\end{align}
Now that we know all the Euclidean distances traveled by the vehicle in the fuel mode through variables $z_{uv}$ and $p_v$, we can state the objective of the path planning problem.
\begin{align}
    \text{minimize} \quad \sum_{(u,v) \in \mathcal{E}} z_{uv} y_{uv} + \sum_{v \in \mathcal{V}_s} p_{v} w_{v} \label{objective} 
\end{align}
The MIBP aims to optimize \eqref{objective} subject to the requirements in \eqref{source_sink_constraint}-\eqref{d_bound_same_side}. 
This formulation includes a non-convex, bilinear objective and bilinear constraints, making it challenging to solve. In the next section, we reformulate this problem as a MICP, which can be effectively solved using commercial solvers.

\section{Reformulated mixed Integer Conic Program (R-MICP)}\label{sec:RMICP} 
The key idea in the reformulation is to address the non-convexity caused by bilinear terms by introducing new decision variables to replace these bilinear terms. To this end, the new decision variables are defined as follows:
\begin{align}
a_{uv} &= q_{u_{o}} y_{uv}  \quad \text{and} \quad a_{uv}^{'} = q_{v_{e}} y_{uv}      \quad \forall (u,v) \in \mathcal{E}, \label{sub_a} \\
b_{uv} &= z_{uv} y_{uv}  \quad \text{and}  \quad l_{uv} = L_{uv} y_{uv}  \quad \forall (u,v) \in \mathcal{E} \label{sub_z}, \\
c_{uv} &= \lambda_{u_{o}} y_{uv} \quad \text{and} \quad c_{uv}^{'} = \lambda_{v_{e}} y_{uv}   \quad \forall (u,v) \in \mathcal{E}, \label{sub_c} \\
e_{v} &= p_{v} w_{v}  \ \ \  \quad   \text{and}  \quad f_{v} = d_{v} w_{v}    \ \ \ \quad  \forall v \in \mathcal{V}_s. \label{sub_same}
\end{align}
Subsequently, all the equations in  \eqref{source_sink_constraint}-\eqref{objective} will be reformulated
in terms of decision variables $a_{uv}$, $a_{uv}^{'}$, $c_{uv}$, $c_{uv}^{'}$, $b_{uv}$, $l_{uv}$, $e_{v}$, $f_{v}$ and $y_{uv}$. We will first start by reformulating the main constraints and then proceed to the flow constraints in the problem.

\subsection{\it Reformulated SOC constraints}  
By substituting \eqref{sub_a} and \eqref{sub_z} into the bilinear constraints \eqref{SOC_constraints_o} and \eqref{SOC_constraints_i}, a set of linear constraints in terms of $a_{uv}$, $a_{uv}^{'}$, $b_{uv}$, and $l_{uv}$ can be derived. For any inter-quiet zone edge in $\mathcal{E}_o$, the reformulated constraints are given as:  
\begin{align}
        & a_{uv}^{'} = a_{uv} + \beta b_{uv} - \alpha (l_{uv} - b_{uv}) \quad \ \  \forall (u,v) \in \mathcal{E}_o. \label{SOC_constraints_o_new}
\end{align}
Similarly, for any intra-quiet zone edge in $\mathcal{E}_{in}$, the corresponding reformulation is:  
\begin{align}
        & a_{uv}^{'} = a_{uv} - \alpha l_{uv} \quad \forall (u,v) \in \mathcal{E}_{in}. \label{SOC_constraints_i_new}
\end{align}
Boundary conditions for the decision variables $a_{uv}$ and $a_{uv}^{'}$ are imposed to ensure their deactivation, i.e., they are set to zero when the corresponding path variable $y_{uv}$ is deactivated ($y_{uv} = 0$). These boundary constraints are derived by multiplying \eqref{q_bound} by $y_{uv}$ and substituting \eqref{sub_a}. The resulting constraints are as follows:
\begin{equation}
\begin{aligned}
    q_{\text{min}} y_{uv} &\leq a_{uv} \leq q_{\text{max}} y_{uv} \quad  \forall (u,v) \in \mathcal{E} \\
    q_{\text{min}} y_{uv} &\leq a_{uv}^{'} \leq q_{\text{max}} y_{uv} \quad \forall (u,v) \in \mathcal{E}.
\end{aligned}
\label{q_bound_new}
\end{equation}

\subsection{\it Reformulated path boundary constraints:} 
Boundary conditions for the decision variables $c_{uv}$ and $c_{uv}^{'}$ are imposed to ensure their deactivation, i.e., they are set to zero when the corresponding path variable $y_{uv}$ is deactivated ($y_{uv} = 0$).  To this end, the boundary constraints for $c_{uv}$ and $c_{uv}^{'}$ can be obtained by substituting \eqref{sub_c} in \eqref{Lambda_bound} as follows:
\begin{align}
     \underline{\lambda}^{uv}_u y_{uv} \leq  c_{uv} \leq \overline{\lambda}^{uv}_u y_{uv} \quad \forall (u,v) \in \mathcal{E}. \notag \\
     \underline{\lambda}^{uv}_v y_{uv} \leq  c_{uv}^{'}  \leq
    \overline{\lambda}^{uv}_v y_{uv}    \quad \forall (u,v) \in \mathcal{E}. \label{c_bound} 
\end{align}
Similarly, boundary constraints when travelling from the source $s$ to a side $v$ or from a side $u$ to the goal $g$ is as follows:
\begin{align}
  \underline{\lambda}^{sv}_v y_{sv} \leq  c_{sv} & \leq
    \overline{\lambda}^{sv}_v y_{sv} \quad \forall (s,v) \in \mathcal{E}. \\
    \underline{\lambda}^{ug}_u y_{ug} \leq  c_{ug} & \leq \overline{\lambda}^{ug}_u y_{ug} \quad \forall (u,g) \in \mathcal{E}.  \label{eq:c_source_goal_bound} 
\end{align}

\subsection{\it Reformulated distance minimization in fuel mode constraints}
The constraint \eqref{eq:intra_length_constraint} on the travel distance between two sides can be expressed in terms of new decision variables \( c_{uv} \), \( c_{uv}^{'} \) and $l_{uv}$ (\eqref{sub_z} and \eqref{sub_c}). This reformulation results in \eqref{eq:intra_length_constraint_c}.
\begin{equation}
\begin{split}
    l_{uv}^2 \geq \big\| n_u y_{uv} + c_{uv}(m_u - n_u) - n_v y_{uv} \\
    \quad  - c_{uv}^{'} (m_v - n_v) \big\|_2^2, \quad \forall (u,v) \in \mathcal{E}.
\end{split}
\label{eq:intra_length_constraint_c}
\end{equation}
The boundary constraints for the new variable $b_{uv}$ can be obtained by multiplying \eqref{z_bound} by \( y_{uv} \) and substituting \eqref{sub_z}:
\begin{align}
        0 \leq  b_{uv} & \leq l_{uv} \quad  \forall (u,v) \in \mathcal{E}.    \label{z_bound_new}
\end{align}
Similarly, the boundary constraints for variables associated with travelling along same side can be obtained by multiplying \eqref{d_bound_same_side} by $w_v$ and substituting \eqref{sub_same} :
\begin{align}
    0 \leq e_v \leq f_v, \ \ 0  \leq f_v  \leq & \|m_v-n_v\|_2 w_{v} \quad  \forall v \in \mathcal{V}_s.\label{z_bound_same_side_new}
\end{align}

The bilinear objective function then can be reformulated as a linear objective function after substituting \eqref{sub_z} in the objective function \eqref{objective} as given in: 
\begin{align}
    \text{minimize} \quad \sum_{(u,v) \in \mathcal{E}} b_{uv} + \sum_{v \in \mathcal{V}_s} e_{v}\label{objective_new} 
\end{align}

\subsection{ \it Reformulated flow constraints}
The flow constraints \eqref{source_sink_constraint}-\eqref{flow_constraints_max_one} associated with the path variable $y_{uv}$ remain valid. However, additional valid inequalities associated with new decision variables $a_{uv}$, $a_{uv}^{'}$, $c_{uv}$, and $c_{uv}^{'}$ are added that follow directly from the flow constraints. Given the initial charge \( q_s \), the constraint for the vehicle's initial SOC is obtained by multiplying \eqref{source_sink_constraint} by \( q_s \) and applying the appropriate substitutions, resulting in \eqref{a_source}. 
\begin{align}
     \sum_{(s,v) \in E} & a_{sv} = q_{s} \quad \quad \forall (s,v) \in \mathcal{E}. \label{a_source} 
\end{align}
Similarly, the constraint related to the vehicle's final SOC is obtained by multplying \eqref{source_sink_constraint} with \( q_g \). After necessary substitutions, this results in the following inequality:
\begin{align}
    q_{\text{min}}  \leq \sum_{(u,g) {\in} E} & a_{ug}^{'} \leq q_{\text{max}} \quad \forall (u,g) \in \mathcal{E}.     \label{a_sink} 
\end{align}
For flow constraints associated with $a_{uv}$, $a_{uv}^{'}$, $c_{uv}$ and $c_{uv}^{'}$ the entry and exit flows at a node may differ due to the possibility of different entry and exit points along a side. This implies that the following inequalities can be true:
\begin{align}
    \sum_{(u,v) \in \mathcal{E}} a_{uv}' & \neq \sum_{(v,u) \in \mathcal{E}} a_{vu}, \quad \forall v \in \mathcal{V} \setminus \{s, g\} \\
    \sum_{(u,v) \in \mathcal{E}} c_{uv}' & \neq \sum_{(v,u) \in \mathcal{E}} c_{vu}, \quad \forall v \in \mathcal{V} \setminus \{s, g\}. \label{flow_constraints_not}
\end{align}

where \( \sum_{(u,v) \in \mathcal{E}} a_{uv}' \) and \( \sum_{(u,v) \in \mathcal{E}} c_{uv}' \) represent the entry flows for SOC and position at node \( v \), respectively, while \( \sum_{(u,v) \in \mathcal{E}} a_{vu} \) and \( \sum_{(u,v) \in \mathcal{E}} c_{vu} \) represent the corresponding exit flows at node \( v \). To establish a relationship between the entry and exit flows for both SOC and position at a node, Theorem~\eqref{thm:flow_constraint} is derived.

\begin{theorem}[Flow constraints for the new variables] \label{thm:flow_constraint}
The flow constraints for the auxiliary variables \( a_{uv} \), \( a_{uv}' \), \( c_{uv} \), and \( c_{uv}' \), derived from the constraints \eqref{SOC_constraints_i}--\eqref{flow_constraints}, are given as:
\begin{subequations} \label{Flow_constraint_a_c}
\begin{align}
    \beta e_{v} {-} \alpha (f_{v} {-} e_{v}) {=} & \sum_{(u,v) \in \mathcal{E}} a_{vu}' {-} \sum_{(u,v) \in \mathcal{E}} a_{uv}, \ \forall v \in \mathcal{V}_s. \label{flow_constraint_SOC} \\
    S_v = \sum_{(u,v)  \in \mathcal{E}} & c_{uv} - \sum_{(u,v) \in \mathcal{E}} c_{vu}', \ \   \forall v \in \mathcal{V}_s.  \label{Sv_definition} \\
    S_{\text{abs}}^v \geq S_v, \ \ S_{\text{abs}}^v   \geq & -S_v,  \ \  S_{\text{abs}}^v \geq 0, \ \  \forall v \in \mathcal{V}_s.  \label{Sabs_constraints} \\
    S_{\text{abs}}^v &= \frac{f_{v}}{\| n_v - m_v \|_2}, \ \ \forall v \in \mathcal{V}_s.  \label{Sabs_definition}
\end{align}
\end{subequations}
\end{theorem}

\begin{proof}
To establish the relationship between entry and exit flow for each node \( v \in \mathcal{V}_s \), we begin with the SOC constraints defined by \eqref{SOC_constraints_same} for the entry and exit from a side.  Using \eqref{flow_constraints} and \eqref{flow_constraints_max_one}, we substitute for \( w_{v} \), resulting in:
\begin{align}
    q_{v_o} \sum_{(v,u) \in \mathcal{E}} y_{vu} = & q_{v_e} \sum_{(u,v) \in \mathcal{E}} y_{uv} \\ \notag
    & + \left( \beta p_v - \alpha (d_v - p_v) \right) w_{v} \ \ \forall v \in \mathcal{V}_s. \label{flow_constraints_a_intermediate}
\end{align}
By substituting for \( a_{uv},  
 a_{uv}', e_{v}, \text{ and } f_{v} \) as defined in \eqref{sub_a} and \eqref{sub_same}, we obtain definition aligned with \eqref{flow_constraint_SOC} as:
\begin{equation}
    \sum_{(v,u) \in \mathcal{E}} a_{vu}' {-} \sum_{(u,v) \in \mathcal{E}} a_{uv} {=} \beta e_{v} {-} \alpha (e_{v} {-} f_{v}),  \ \forall v \in \mathcal{V}_s \label{flow_constraints_a_inter_subs}
\end{equation}
For the length constraint \eqref{length_constraint_same_side}, take the square root and multiply both sides by \( w_{v} \), and using \eqref{flow_constraints} and \eqref{flow_constraints_max_one}, we derive:
\begin{align}
    \left| \lambda_{v_o} \sum_{(v,u) \in \mathcal{E}} y_{vu} - \lambda_{v_e}\sum_{(u,v) \in \mathcal{E}} y_{uv} \right| &\leq \frac{d_v w_{v}}{\| m_v - n_v \|_2} \quad \forall v \in \mathcal{V}_s.  \label{length_constraint_same_side_inter_1} \\
    \left| \sum_{(u,v) \in \mathcal{E}} c_{vu} - \sum_{(u,v) \in \mathcal{E}} c_{uv}' \right| &\leq \frac{f_v}{\| m_v - n_v \|_2} \quad \forall v \in \mathcal{V}_s.  \label{length_constraint_same_side_inter_2}
\end{align}

\noindent To linearize the absolute value constraint in \eqref{length_constraint_same_side_inter_2}, we introduce auxiliary variables \( S_v \) and \( S_{\text{abs}}^v \), defined as:
\begin{align}
    S_v &= \sum_{(u,v) \in \mathcal{E}} c_{vu} - \sum_{(u,v) \in \mathcal{E}} c_{uv}' \notag \quad \quad \quad \forall v \in \mathcal{V}_s. \\
    S_{\text{abs}}^v &\geq S_v, \quad S_{\text{abs}}^v \geq -S_v, \quad S_{\text{abs}}^v \geq 0 \quad \forall v \in \mathcal{V}_s. \notag \\
    S_{\text{abs}}^v &= \frac{f_v}{\| m_v - n_v \|_2} \quad \quad \quad \quad \quad \quad \quad \ \ \forall v \in \mathcal{V}_s. \notag
\end{align}
\noindent These definitions align directly with the constraints in \eqref{Flow_constraint_a_c}, thus completing the proof.
\end{proof}

\subsection{Reformulated problem}
The R-MICP aims to optimize \eqref{objective_new}, subject to the constraints \eqref{SOC_constraints_i_new}--\eqref{Flow_constraint_a_c} and the flow constraints \eqref{source_sink_constraint}--\eqref{flow_constraints_max_one}. This problem can be efficiently solved using commercial solvers. Once solved, the values of \( c_{uv} \), \( c_{uv}' \), \( a_{uv} \), \( a_{uv}' \), and \( y_{uv} \) can be used to compute the parameters \( \lambda_{u_o} \), \( \lambda_{v_e} \), \( q_{u_o} \), and \( q_{v_e} \) for the nodes along the optimal path indicated by \( y_{uv} \). Consequently, the optimal path, the switching points, and the SOC at each switching point for the HUAV are determined. Moreover, by relaxing the integrality constraints on \( y_{uv} \) and \( w_{v} \), a tight lower bound can also be computed efficiently.

\section{Extension to the Traveling Salesman Problem (TSP) with intermediate quiet zones}\label{sec:TSP}


In the TSP version of the problem, the HUAV visits each of the given target locations (or targets) at least once before returning to its source. In the standard TSP, the cost between targets is typically predefined or easily calculated, often using Euclidean distances. However, the presence of intermediate \enquote{quiet zones} introduces a dependency of the cost on the sequence of visits in the TSP. This dependency arises because the cost is influenced by the initial and final SOC when traveling from one target to another. As a result, the objective is no longer limited to finding the optimal tour; it also involves simultaneously determining the optimal path, including mode-switching points, between consecutive targets.

To solve the TSP for the HUAV, we first define a graph \(G_{\text{TSP}} := (\mathcal{N}, \mathcal{L})\)
where \( \mathcal{N} \) represents the set of target nodes, and \( \mathcal{L} \) denotes the edges, which correspond to possible paths between the targets. For each edge \( e = (i, j) \in \mathcal{L} \), let \( x_{ij} \) be a binary decision variable such that \( x_{ij} = 1 \) if the edge is traversed, and \( x_{ij} = 0 \) otherwise. Noting that the goal is to construct a tour with minimum fuel consumption, the objective function is given by
\begin{align}
    \text{Minimize:} & \quad \sum_{(i,j) \in \mathcal{L}} \left( \sum_{(u,v) \in \mathcal{E}} b_{uv} + \sum_{v \in \mathcal{V}} g_{v} \right) x_{ij} \label{objective_tsp}
\end{align}
The constraints for the TSP include similar constraints presented in the MICP formulation corresponding to the path problem, which include
\begin{align}
    \text{Constraints: \eqref{SOC_constraints_o_new}--\eqref{Sabs_definition} and \eqref{source_sink_constraint}--\eqref{flow_constraints_max_one},} \label{Old_formulation}
\end{align}
for each pair of targets (including the source). Additionally, since all targets need to be visited once, two additional constraints are introduced for the TSP: one constraint to ensure that each target is visited, and another constraint to ensure that subtours are eliminated from the candidate set of tours. The two constraints are given by
\begin{align}
    \sum_{(i,j) \in \mathcal{L}} x_{ij} = \sum_{(j,i) \in \mathcal{L}} x_{ji} = 1 \quad \forall j \in \mathcal{N}, \label{Flow_constriants_TSP} \\
    \sum_{\substack{j \in Q \\ j \neq i}} x_{ij} \leq |Q| - 1, \quad \forall Q \subseteq \mathcal{N}, \; |Q| \geq 2, \label{Subtour_elimination}
\end{align}

However, solving the TSP, which requires simultaneous optimization of both the target sequence and the paths between consecutive targets, becomes computationally intractable as the map size and the number of targets increase. To address this, we propose two approaches to decouple the optimization of the TSP tour from the determination of optimal paths between consecutive targets.

\subsection{Generalized TSP approach}
In this approach, the state of charge (SOC) at each target \( i \in \mathcal{N} \) is discretized into \( d \) levels, forming a set \( S_i \) of nodes corresponding to target $i$. Here, each node in $S_i$ corresponds to a specific SOC level. The objective is to compute a tour of minimum cost that visits exactly one node from each set \( S_i \). It should be noted here that discretization of the SOC allows us to compute the cost between any two nodes selected from sets $S_i, S_j,$ $i \neq j,$ independent of the sequence in which the targets are visited in the tour. To this end, we utilize our proposed MICP formulation to compute the cost between every pair of nodes, since our MICP formulation enables the specification of the starting SOC and goal SOC. 

The resulting problem, wherein one node from each set $S_i$ needs to be visited, and the vehicle must return to its source, yields the Generalized TSP (GTSP) \cite{pop2024comprehensive}.  
To solve the GTSP, the problem is first transformed into an Asymmetric TSP (ATSP) using the Noon-Bean transformation \cite{noon1993efficient}. The resulting ATSP is then solved using the LKH heuristic \cite{helsgaun2009general}, which yields a near-optimal TSP tour. Using the tour constructed from the ATSP, the corresponding tour in the GTSP can be obtained \cite{noon1993efficient}, wherein one node in each set $S_i$ is visited. Therefore, the optimal sequence to visit the targets and the corresponding SOC level at each target are obtained. 

Once the sequence of targets to be visited is determined (or the binary variables are given), the MIBP formulation in Section \ref{sec:MIBP} is used to find the optimal path and mode-switching points between consecutive targets, thereby providing the complete TSP solution for the HUAV. However, this approach requires \(n(n-1)d^2\) number of cost computations where \( n \) represents the total number of targets, making this approach computationally expensive for large-scale problems. This will be demonstrated using comprehensive numerical results that will follow in the next section. 

\subsection{ Minimum SOC Approach}
To reduce the computational effort in solving the GTSP, the second approach assumes the HUAV departs each target node with a SOC at least equal to \( q_{\text{min}} \). This assumption is reasonable, as reaching a target with minimal SOC is sufficient to find a feasible solution for the HUAV. 
While the SOC of the vehicle is set to $q_{max}$ at departure from the source, we compute the travel costs between all the targets by assuming the SOC of the vehicle is equal to $q_{min}$ when it departs from any target. Once these costs are computed, the problem reduces to a standard TSP, which is solved using the LKH heuristic. This approach requires only \( n^2 \) travel cost computations, making it more efficient than the GTSP-based approach. After determining the optimal target sequence using LKH, the MIBP formulation (section \ref{sec:MIBP}) is used to find the optimal path and mode-switching points between consecutive targets, providing us with the complete TSP solution for the HUAV.
\section{Computational results}
\label{sec:Results}
In this section, computational results are provided to demonstrate the effectiveness of the proposed algorithm for planning the path between two given locations and for constructing a tour to visit a given set of locations. To this end, the formulated MICP was implemented under two configurations. In the first configuration, the edge selection variable \( y_{uv} \) for any $(u,v)\in \mathcal{E}$ is relaxed to lie in the interval $[0, 1]$ to provide a lower bound. On the other hand, in the second configuration, \( y_{uv} \) is set as a binary variable to yield the optimal cost and path. 

In this study, four distinct maps were used for testing. The first map was manually designed as a dense environment containing fifteen quiet zones, some of which were non-convex. The second map was generated by randomly placing fifteen polygonal quiet zones within a $12000\times 8000$ unit area.  Additionally, two real-world maps based on New York and Boston were generated, as discussed in \cite{manyam2022path}.

It should be noted here that all tests were performed on a laptop with a 12th-generation Intel i7-12700H processor and 32 GB of RAM. The algorithms were implemented in Python, and Gurobi \cite{gurobi} was used to solve the formulated mathematical programs.

 \subsection{Path planning between two locations}

In this section, we present our results for planning the path between two given locations. Furthermore, we compare our results with the current state-of-the-art approach described in \cite{manyam2022path}.
It should be recalled here that in \cite{manyam2022path}, the sides of the quiet zones and the SOC of the battery are discretized. Henceforth, we will refer to the algorithm in \cite{manyam2022path} as a ``discrete" approach.
To this end, we utilize the four maps discussed previously. For each map, 50 scenarios were tested, with randomly chosen start and goal nodes separated by an Euclidean distance exceeding a specified threshold of 1000 unit. The discharge rate \( \alpha \) and recharge rate \( \beta \) were set to 0.08 and 0.04 for the HUAV, respectively. The minimum SOC \( q_{\text{min}} \), maximum SOC \( q_{\text{max}} \), and initial SOC \( q_{\text{init}} \) were set to 20\%, 100\%, and 100\%, respectively. In the implementation of the discrete approach, we discretized the SOC levels between $q_{min}$ and $q_{max}$ into ten levels, and side lengths are discretized at points spaced 100 units apart. This approach is used to calculate an upper bound for the problem.

\textbf{Remark:} All maps were scaled to ensure their dimensions are within 100×100 units, to mitigate numerical issues encountered by the solver. To this end, all coordinates were scaled by a factor \( \gamma < 1 \), reducing the coordinate range to within 100 units. Following \eqref{SOC_constraints_o}, the discharge and recharge rates were scaled by \( 1/\gamma \), and the cost output was adjusted by the same factor. Additionally, to enhance computational efficiency, the program was terminated when the optimality gap was within $1\%$ or if the solver (Gurobi) ran for more than 5 seconds.

\begin{figure*}[h!]
    \centering
    \begin{subfigure}[t]{0.3\textwidth}
        \centering
        \includegraphics[width=\textwidth]{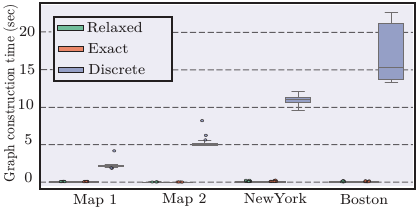}
        \caption{Graph construction time\vspace{0.5\baselineskip}}
        \label{subfig:graph_construction_time}
    \end{subfigure}
    \hfill
    \begin{subfigure}[t]{0.3\textwidth}
        \centering
        \includegraphics[width=\textwidth]{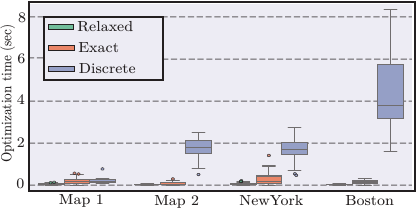}
        \caption{Optimization time for all three approaches}
        \label{subfig:optimization_time_three_approaches}
    \end{subfigure}
    \hfill
    \begin{subfigure}[t]{0.3\textwidth}
        \centering
        \includegraphics[width=\textwidth]{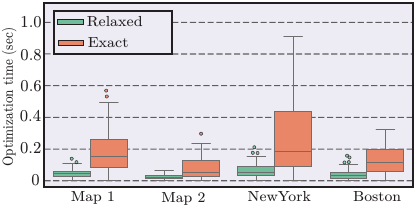}
        \caption{Optimization time comparison between relaxed and exact approach \vspace{0.5\baselineskip}}
        \label{subfig:optimization_time_two_approaches}
    \end{subfigure}
    \caption{Computational times for different maps and algorithms.}
    \label{fig:Time_taken}
\end{figure*}

\begin{figure}[htb!]
    \centering
    \begin{subfigure}[b]{0.8\columnwidth}
        \centering
        \includegraphics[width=\linewidth]{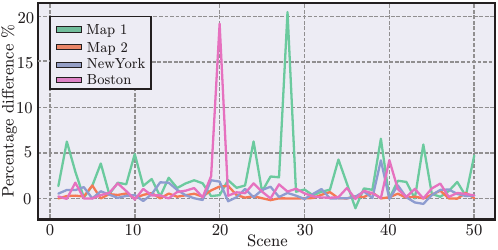}
        \caption{Percentage difference between exact solution's cost ($C_e$) and discrete solution's cost ($C_d$), defined as $100 \left(\frac{C_{\text{d}}-C_{\text{e}}}{C_{\text{d}}} \right)$.}
        \label{subfig:percentage_diff_exact_vs_discrete}
    \end{subfigure}
    \vskip 0.5em 
    \begin{subfigure}[b]{0.8\columnwidth}
        \centering
        \includegraphics[width=\linewidth]{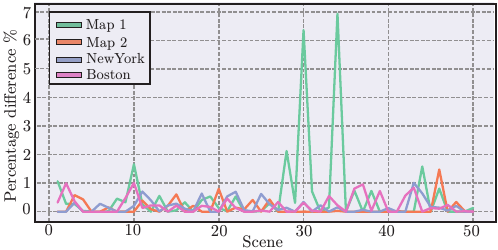}
        \caption{Percentage difference between relaxed solution's cost ($C_r$) and exact solution's cost ($C_e$), defined as $100 \left(\frac{C_{\text{e}}-C_{\text{r}}}{C_{\text{e}}} \right)$.}
        \label{subfig:percentage_diff_continous_vs_exact}
    \end{subfigure}
    \caption{Percent difference in cost of solutions for all the maps.}
    \label{fig:percent_diff}
\end{figure}

The results of the simulations are presented in Figs. \ref{fig:Time_taken}-\ref{fig:path_with_charge}, and Table \ref{tab:Time_table}. In Fig.~\ref{fig:Time_taken}, the time taken for graph construction and optimization for all three approaches is shown. From this figure, we can make the following observations:
\begin{enumerate}
    \item From Fig.~\ref{subfig:graph_construction_time}, we can observe that the time to construct the graph for the discrete approach is significantly higher than the two approaches we have proposed. This is due to the large number of nodes generated by discretization of sides of quiet zones and SOC, leading to an increase in computation time by nearly 100 times.
    \item Furthermore, from Table~\ref{tab:Time_table} and Fig.~\ref{subfig:optimization_time_three_approaches}, we can observe that across all maps, the relaxed formulation required substantially less optimization time (by an order of magnitude) compared to the discrete approach. For the exact formulation, optimization time on Map 1 was comparable to the discrete approach, while for the other maps, it was lower by an order of magnitude.
    \item From Fig.~\ref{subfig:optimization_time_two_approaches}, we can also observe that the relaxed formulation took significantly less time for optimization (by around 4 to 5 times) compared to the exact formulation in all the maps.
\end{enumerate}
From these observations, we can observe that our proposed approaches outperform the discrete approach in terms of computation time.

\begin{table}[h!]
\centering
\caption{Average Optimization time (opt) and Average Graph Construction time (Graph) in secs for different algorithms and maps.}
\resizebox{\columnwidth}{!}{
\begin{tabular}{|c|p{0.6cm}|p{0.5cm}|p{0.5cm}|p{0.5cm}|p{0.5cm}|p{0.6cm}|p{0.5cm}|p{0.6cm}|}
\hline
 & \multicolumn{2}{c|}{Map 1} & \multicolumn{2}{c|}{Map 2} & \multicolumn{2}{c|}{New York} & \multicolumn{2}{c|}{Boston} \\ \cline{2-9}
                  & Opt & Graph                   & Opt  & Graph            & Opt  & Graph            & Opt  & Graph      \\ \hline
Relaxed        & 0.046    & 0.019          & 0.024    & 0.041          & 0.065    & 0.020          & 0.044    & 0.034   \\ \hline
Exact            & 0.192    & 0.019          & 0.076    & 0.047          & 0.302    & 0.021          & 0.178    & 0.035    \\ \hline
Discrete          & 0.205    & 2.169          & 1.753    & 5.161          & 1.731    & 17.355          & 4.296    & 10.968   \\ \hline
\end{tabular}
}
\label{tab:Time_table}
\end{table}
Now, we desire to compare the cost of the solution obtained from the discrete approach and the two approaches (exact and relaxed) that we have proposed. A comparison of the same for the four maps is shown in Fig.~\ref{fig:percent_diff}. From this figure, we can make the following observations:
\begin{enumerate}
    \item From Fig.~\ref{subfig:percentage_diff_exact_vs_discrete}, we can observe that the exact cost is consistently lower than the upper bound cost obtained via the discrete approach, with a single exception observed in Map 1. The average percentage differences across maps are \(2.19\%\) for Map 1, \(0.31\%\) for Map 2, \(0.58\%\) for New York, and \(1.1\%\) for Boston. Therefore, we can conclude that the discrete approach is marginally suboptimal compared to the exact formulation. Hence, coupled with the increased optimization time requirement for the discrete approach, we can observe the effectiveness of our proposed MICP formulation over the discrete approach.
    \item From Fig.~\ref{subfig:percentage_diff_continous_vs_exact}, we can observe that in all fifty scenarios across different maps (except for two scenarios in Map 1), the percentage gap between the relaxed and exact formulations remains below \(2\%\).
    Furthermore, the average percentage differences for Map 1, Map 2, New York, and Boston are \(0.55\%\), \(0.12\%\), \(0.21\%\), and \(0.17\%\), respectively. Since the cost obtained from the exact formulation is within an optimality gap of \(1\%\), we can conclude that relaxing the edge selection variables \(y_{uv}\) provides a tight lower bound.
\end{enumerate}
Therefore, from these observations, we can conclude that we can obtain a tight lower bound and a near-optimal solution from the formulation that we proposed. 


It should be noted that in addition to obtaining the cost of the path traversed by the HUAV using the exact formulation, we can also obtain the optimal path, switching points, and the SOC at each switching point, as depicted in Fig.~\ref{fig:path_with_charge} for Map 1. In this scenario, the HUAV travels along the boundary of one of the quiet zones because its SOC is insufficient for directly crossing the quiet zone using a straight line segment. 


\begin{figure}[htb!]
 \begin{subfigure}[b]{0.8\columnwidth}
     \centering
     \includegraphics[width=\linewidth]{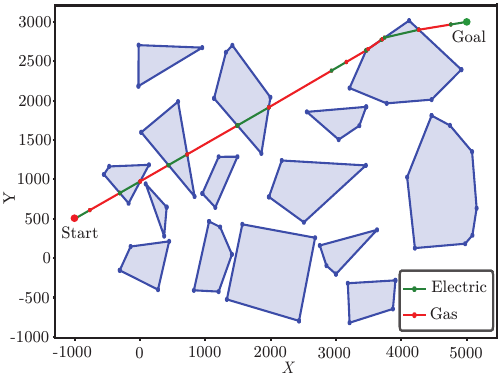}
     \caption{Depiction of the path of the HUAV along with the switching points between electric and fuel modes}
     \label{subfig:path_final}
 \end{subfigure}
 \hfill
 \begin{subfigure}[b]{0.8\columnwidth}
     \centering
     \includegraphics[width=\linewidth]{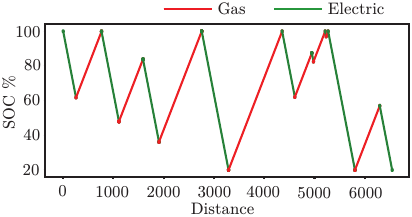}
     \caption{Depiction of the SOC of the HUAV as it charges and recharges as a function of the distance traversed}
     \label{subfig:path_charge}
 \end{subfigure}
    \caption{Demonstration of path obtained from solving the MICP of the HUAV and the change in SOC for Map 1}
    \label{fig:path_with_charge}
\end{figure}

\subsection{TSP with intermediate quiet zones}

Having observed the effectiveness of our approach for planning the path between two locations, we now present the results for the TSP variant of the problem with quiet zones in this section.
To this end, numerical results are presented for both approaches: (1) the GTSP method and (2) the minimum SOC assumption method. For these experiments, the discharge rate \(\alpha\) and the recharge rate \(\beta\) were fixed to be $0.1$ and $0.05,$ respectively. Furthermore, multiple instances were generated for varying numbers of target points across different maps. A total of nearly 240 instances, each with different maps and target nodes, were constructed.

\begin{figure*}[htb!] 
 \begin{subfigure}[t]{0.45\textwidth}
     \centering
     \includegraphics[width=\textwidth]{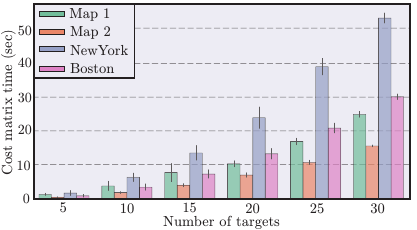}
     \subcaption{Time taken for cost matrix computation.}
     \label{subfig:cost_matrix_time}
     \hspace{0.05\linewidth}
 \end{subfigure}
 \hfill
    \begin{subfigure}[t]{0.45\textwidth}
     \centering
     \includegraphics[width=\textwidth]{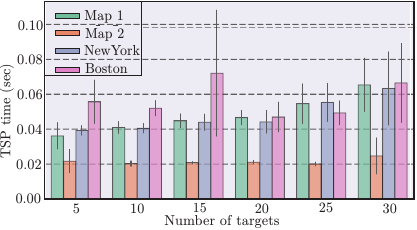}
     \subcaption{Time taken for TSP solution}
     \label{subfig:TSP_time}
 \end{subfigure}
 \hfill
  \begin{subfigure}[t]{0.45\textwidth}
     \centering
     \includegraphics[width=\textwidth]{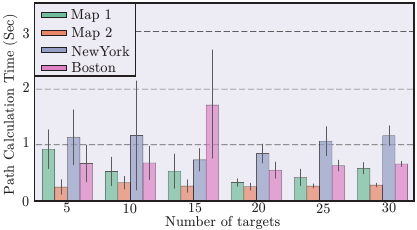}
     \subcaption{Time taken to compute optimal path and mode switching points between the targets in sequence}
     \label{subfig:path_calc_time}
     \hspace{0.05\linewidth}
 \end{subfigure}
 \hfill
   \begin{subfigure}[t]{0.45\textwidth}
     \centering
     \includegraphics[width=\textwidth]{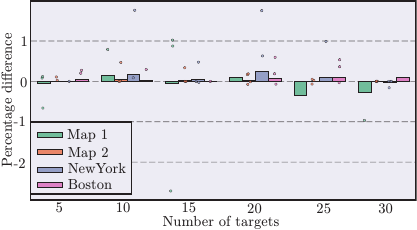}
     \subcaption{Percentage difference between the costs obtained using the GTSP method and the min SOC approach.}
     \label{subfig:diff_two_TSP_methods}
 \end{subfigure}
    \caption{Results for the minimum SOC assumption method and its comparison with GTSP method}
    \label{fig:Result_TSP}
\end{figure*}

\begin{figure}[htb!]
     \centering
     \includegraphics[width=\linewidth]{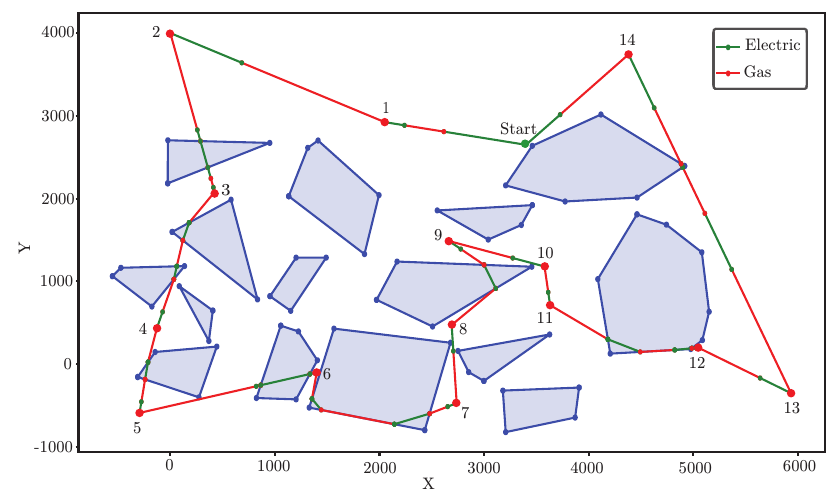}
     \caption{The path found using the \(q_{\text{min}}\) assumption approach is presented, with the numbers indicating the sequence of visits.}
     \label{fig:TSP_path}
 \end{figure}
 
Noting from Table~\ref{tab:Time_table} that the relaxed MICP formulation requires 4-5 times less computation time compared to the exact MICP approach, we use relaxed MICP formulation to calculate the node-to-node cost matrix required to solve the TSP. Once the tour is determined, the exact paths between the targets can be obtained by applying the exact MICP formulation to each consecutive pair of targets in the sequence. In Figs.~\ref{subfig:cost_matrix_time}, \ref{subfig:TSP_time}, and \ref{subfig:path_calc_time}, we show the computation time for the second approach, wherein the HUAV starts from \(q_{\text{min}}\) at each target. 
From these figures, we can make the following observations:
\begin{enumerate}
    \item From Fig.~\ref{subfig:cost_matrix_time}, we can observe that the computation time for computing the cost matrix increases with increasing number of targets. However, we can observe that the maximum computation time is around 1 minute. In contrast, for the GTSP method, the computation time to generate the cost matrix for Map 1, Map 2, New York, and Boston Maps were 1980, 2290, 5117, and 3278 seconds, respectively, for 31 targets. It should be noted that for the GTSP implementation, for instances with a lesser number of targets, we select a subset of targets randomly from the generated 31 targets. In this regard, we compute the cost matrix once per map for the GTSP implementation. Nevertheless, we can observe that the cost matrix computation is substantially more expensive than the method involving the assumption that the HUAV SOC is $q_{min}$ at each target.

    \item In Fig.~\ref{subfig:TSP_time}, the computation time for the TSP is shown for an increasing number of targets. It can be observed that the computation time for computing the tour is less than a second due to the use of LKH \cite{helsgaun2009general}, which is a heuristic that yields a near-optimal tour for very large instances of TSP.

    \item In Fig.~\ref{subfig:path_calc_time}, the path computation time is shown. From this figure, we can observe that the computation time does not increase with the number of targets since, in a fixed map size, increasing the number of targets leads to closer pairs, simplifying optimization. Conversely, fewer targets require traversing more quiet zones, increasing computational effort.
\end{enumerate}

In addition, we desire to compare the solution cost for the TSP from the two discussed approaches. A summary of this comparison is shown in Fig.~\ref{subfig:diff_two_TSP_methods}. In this figure, the percentage cost difference, defined as \( \left( \frac{C_A - C_G}{C_G}  \right) \), between the GTSP method (\(C_G\)) and the minimum SOC assumption method (\(C_A\)) is shown. From this figure, we can observe that the average percent difference is small, with the maximum difference not exceeding \(2\%\). Therefore, the assumption of the HUAV starting at \(q_{\text{min}}\) at each target is reasonable to reduce the computation time without compromising on the solution quality. A demonstration of the tour obtained using the $q_{min}$ assumption method is shown in Fig.~\ref{fig:TSP_path}. In this figure, the tour of the HAUV with the sequence in which the targets need to be visited, the path between each pair of targets, and the switching points are shown.

\section{Conclusion}\label{sec:conclusions}

In this paper, we addressed the path planning problem for a gas-electric HUAV operating in environments with noise-restricted zones, referred to as quiet zones. This problem is particularly relevant in an environment containing residential areas where noise pollution is strictly regulated. Our proposed algorithm enables the HUAV to traverse these quiet zones using the low-noise electric mode, avoiding suboptimal detours.

In our approach, SOC and position of the HUAV along the sides of the quiet zones were modeled as continuous variables, and the problem was formulated as a mixed integer bilinear program. By introducing auxiliary variables, the bilinear constraints were transformed into a mixed integer conic program, which can be efficiently solved to identify near-optimal path and mode switching points. To validate this formulation, we conducted a comprehensive numerical analysis and compared it with a state-of-the-art discrete approach. Our empirical study on 200 instances highlights the advantages of the exact MICP approach, which achieved an average travel cost reduction of around 1\% over the discrete approach, alongside a tenfold reduction in computation time. Additionally, the relaxed MICP formulation provides a tight lower bound with an average gap of only 0.26\% and requires 4.5 times less computation time than the exact MICP formulation.

The proposed MICP was also utilized to solve the TSP for a HUAV in environments with quiet zones. We developed two methodologies: a GTSP approach and a minimum SOC assumption approach. While the GTSP approach yields accurate solutions, it incurs high computational costs due to the large number of cost calculations required. In contrast, the minimum SOC assumption approach, which assumes the HUAV reaches each target with minimum SOC, produced solutions within 1\% of the GTSP results while significantly improving computational efficiency.

In summary, we proposed a new approach to the path planning problem for HUAVs in noise-restricted environments and demonstrated its effectiveness through numerical comparisons with state-of-the-art methods and its application to a TSP variant of the problem.

\section{Acknowledgments}
This work was cleared for public release under APRS-RYA-2025-05-00001 and the views or
opinions expressed in this work are those of the authors and do not reflect any position
or the opinion of the United States Government, US Air Force, or Air Force Research
Laboratory.

\noindent The authors also thank Dr. Satyanarana Gupta Manyam for providing the code for the algorithm developed in \cite{manyam2022path}.

\bibliographystyle{IEEEtran}
\bibliography{References}

\section{Appendix}

{\bf Boundary parameters specification:}
\noindent Given sides $u$ and $v$ belonging to different quiet zones, we explain the heuristic we use to identify the parameters $\underline{\lambda}_{u}^{uv}$, $\overline{\lambda}_{u}^{uv}$, and  $\underline{\lambda}_{v}^{uv}$, $\overline{\lambda}_{v}^{uv}$. To do this, first connect the corners of the side $u$ to the corners the side $v$ using four line segments (refer to Fig.~\ref{fig:Graph cases}). Now, consider the following three cases:

\begin{itemize}
    \item \textbf{Sides are completely visible:} If none of the four line segments connecting the corners intersect any quiet zone boundary as shown in Fig. ~\ref{fig:Fully_Visible}, travel is possible between any point in side $u$ and any point in side $v$. In this case,  $\underline{\lambda}_{u}^{uv}=0$, $\overline{\lambda}_{u}^{uv}=1$, and  $\underline{\lambda}_{v}^{uv}=0$, $\overline{\lambda}_{v}^{uv}=1$.
    
    \item \textbf{Sides are not visible:} If more than one line segment connecting the corners intersects the boundaries of a quiet zone, as illustrated in Fig. ~\ref{fig:Not_visible}, no directly travel is possible between $u$ and $v$. In this case, there is no edge added between the two sides.
    
    \item \textbf{Sides are partially visible:} If exactly one of the four line segments intersects a quiet zone boundary (Fig.~\ref{fig:partial_visibility}, the portions of the sides that can be used for direct travel between them is constrained. In Fig.~\ref{fig:partial_visibility}, we extend the side $v$ to intersect side $u$ at $x_u(\lambda_in)$. Therefore, we set $\underline{\lambda}_{u}^{uv}=\lambda_{in}$, $\overline{\lambda}_{u}^{uv}=1$, and  $\underline{\lambda}_{v}^{uv}=0$, $\overline{\lambda}_{v}^{uv}=1$.
\end{itemize}

\begin{figure}[htb!]
 \begin{subfigure}[b]{0.45\linewidth}
     \centering
     \includegraphics[width=\textwidth]{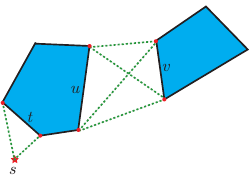}
     \caption{Sides $u$ and $v$ are completely visible.}
     \label{fig:Fully_Visible}
 \end{subfigure}
 \hfill
 \begin{subfigure}[b]{0.45\linewidth}
     \centering
     \includegraphics[width=\textwidth]{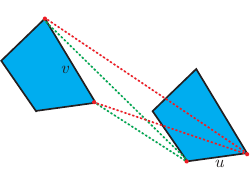}
     \caption{Sides $u$ and $v$ are not visible.}
     \label{fig:Not_visible}
 \end{subfigure}
  \hfill
 \begin{subfigure}[b]{0.45\linewidth}
     \centering
     \includegraphics[width=\textwidth]{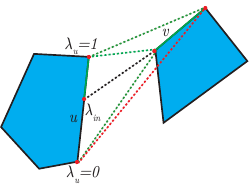}
     \caption{Sides $u$ and $v$ are partially visible.}
     \label{fig:partial_visibility}
 \end{subfigure}
    \caption{Three cases considered for specifying the bounding parameters. The green line segments do not intersect any quiet zone boundaries, whereas the red line segments intersect the quiet zone boundaries.}
    \label{fig:Graph cases}
\end{figure}

\end{document}